\documentclass[11pt,letterpaper]{amsart}

\usepackage{mathrsfs}
\usepackage{cite}
\usepackage{color}
\usepackage{amssymb}
\theoremstyle{plain}


\begin{document}

\numberwithin{equation}{section}
\renewcommand{\theequation}{\arabic{section}.\arabic{equation}}
\theoremstyle{plain}
\newtheorem{exam}{Example}[section]
\newtheorem{theorem}[exam]{Theorem}
\newtheorem{lemma}[exam]{Lemma}
\newtheorem{remark}[exam]{Remark}
\newtheorem{hyp}[exam]{Hypothesis}
\newtheorem{proposition}[exam]{Proposition}
\newtheorem{definition}[exam]{Definition}
\newtheorem{corollary}[exam]{Corollary}
\newtheorem{analytic}[exam]{Analytic extension principle}
\newtheorem{notation}[exam]{Notation}
\setlength{\baselineskip}{1.5\baselineskip}

\title[Finite time blow-up and global existence]{On the existence and nonexistence of global solutions of the semilinear heat equation}
\author[K. Zhang]{Kaiqiang Zhang}
\address{Kaiqiang Zhang, School of Computer Science and Technology, Dongguan University of Technology, Dongguan
523808, China}
\email{mathzhangkq@gmail.com}
\author[Z. Li]{Zhiyu Li$^\ast$}
\address{Zhiyu Li, School of Mathematics and Information Science, Shandong Technology and Business University,
Yantai 264003, China}
\email{202414287@sdtbu.edu.cn}
\thanks{${}^{\ast}$ Corresponding author.}

\begin{abstract}
We consider the semilinear heat equation 
$$
u_t-\Delta u=|u|^{p-1}u,\ \  (t,x)\in\mathbb{R}^+\times\mathbb{R}^n.
$$ 
The well-known difficulty with this problem is that the potential well method cannot be applied directly, due to the scaling invariance which leads to a potential well of zero depth. We employ the forward similarity transform to convert the equation into a new parabolic equation, so that we can apply the potential well method in weighted Sobolev spaces. As a result, we obtain a new criterion that establishes whether solutions to the heat equation blow up in finite time or exist globally. This work extends the partial results of Ikehata et al. (\textit{Ann. Inst. H. Poincaré Anal. Non
Linéaire}, \textbf{27} (2010) 877-900) from critical Sobolev exponent to the case $p_F<p<p_S$, where $p_F=1+2/n$ is the Fujita exponent and $p_S=(n+2)/(n-2)$ (for $n\ge3$) is the critical Sobolev exponent.
\end{abstract}
\maketitle
\section{Introduction}\label{se1}
The semilinear heat equation
\begin{equation}\label{1.1}
\left\{
\begin{aligned}
&u_t-\Delta u=|u|^{p-1}u,\\
&u|_{t=0}=u_0(x),
\end{aligned}
\right.
\quad  (t,x)\in\mathbb{R}^+\times\mathbb{R}^n,\ p>1,\ n\ge3,
\end{equation}
 has been widely studied since the pioneering work of Fujita \cite{Fujita}. This model dissipates the total energy
 $$
 \frac{d}{dt}E(u)<0,\ E(u)=\frac{1}{2}\int_{\mathbb{R}^n}|\nabla u|^2dx-\frac{1}{p+1}\int_{\mathbb{R}^n}|u|^{p+1}dx.
 $$ 
Let $T\in (0,\infty]$ denote the maximal existence time of the solution $u$. If $T<\infty$, the solution blows up in finite time, i.e.,
\begin{equation}\nonumber
\lim_{t\to T}||u(t)||_{L^\infty(\mathbb{R}^n)}=+\infty.
\end{equation}  
From standard parabolic regularity arguments, the Cauchy problem \eqref{1.1} is well posed in
the energy space $H^1(\mathbb{R}^n)$ and for any $u_0\in H^1(\mathbb{R}^n)$, there exists a unique maximal solution 
$u\in C((0, T),H^1\cap L^\infty(\mathbb{R}^n))$, which is moreover smooth on $(0, T)\times \mathbb{R}^n$, see \cite{Quittner-Souplet-2019}.

Studying the finite-time blow-up solutions of parabolic equations has attracted significant attention in recent decades, as it is crucial for understanding the thermal runaway mechanism in combustion theory. Numerous studies have established criteria on initial data that determine whether solutions to \eqref{1.1} blow up in finite time.
The Fujita type result \cite{Fujita,Weissler} shows that the problem \eqref{1.1} admits a global nontrivial solution $u\ge0$ if and only if $p>p_F:=1+2/n$. The critical exponent $p_F$ is known as the Fujita exponent. For the case of negative initial energy, Levine \cite{Levine-1973} showed that solutions must blow up in finite time. Later, Kavian studied \eqref{1.1} using the forward similarity variable in weighted Sobolev spaces and introduced a weighted energy functional, he showed that solutions with negative initial weighted energy also blow up in finite time. For the nonnegative initial data, Lee and Ni \cite{Lee-1992} established another blow-up criterion: if the initial data decays more slowly $|x|^{-2/(p-1)}$ at space infinity, then the solution blows up in finite time. One may refer  \cite{Collot,Wei2021,Collot-2020,Quittner-Souplet-2019} and references therein for other blow up results of \eqref{1.1}.

The global existence of solutions has been a subject of extensive research, as it is fundamental to understanding the long-term dynamics and stability in various physical models.
Fujita \cite{Fujita} proved that if the initial data are dominated by a small multiple of a Gaussian, then the solutions of equation \eqref{1.1} exist globally and remain controlled by the heat kernel. For the case $p>1+2/n$, Lee and Ni \cite{Lee-1992} proved that for any $k>0$, there exists a
small constant $b=b(n,p,k)$ such that if $0\le u_0\le b(k+|x|^2)^{-1/(p-1)}$, the the solution of \eqref{1.1} exists globally. Subsequently, Wang  \cite{Wang-1993} provided an optimal choice of the constant $b$. In the case $p>(n+2)/(n-2)$, Polacik and Yanagida \cite{Yanagida2003} established a sufficient condition for global existence  by using a family of radial positive steady states. Further sufficient conditions for global solutions can be found in \cite{Weissler, Galaktionov-1997,Kavian}.

A powerful tool for studying blow-up solution and global solution  with the parabolic problem is the potential well method, which was established by Sattinger \cite{Payne-1975,Sattinger-1968}. Roughly speaking, by using the energy functional and the Nehari functional, one can define the potential well. The solutions originating from inside the potential well are global, while solutions originating from outside the potential well will blow up in finite time. The potential well method has been used to study the global existence and nonexistence of solutions for various nonlinear evolution equations, see, for example, \cite{Galaktionov-1997,Zhangk1,Quittner-Souplet-2019,Xu20018}. In this paper, we aim to study the existence of blow-up solutions and global solutions for the equation \eqref{1.1} via the potential well method.

The equation \eqref{1.1} admits a scaling invariance: if $u(t,x) $ is a solution, then so is
$u_\lambda(t,x)=\lambda^{\frac{2}{p-1}}u(\lambda^2t,\lambda x),\
\lambda>0.$
The scaling symmetry
$$
u_\lambda(t,x)=\lambda^{\frac{n-2}{2}}u(\lambda^2t,\lambda x),\  E(u_\lambda)=E(u),
$$
reflects the energy critical nature of the problem \eqref{1.1}.
In the energy critical case $p=\frac{n+2}{n-2}$, there are many results concerning the global existence and blow-up to the equation \eqref{1.1}.  Ishiwata and Suzuki \cite{Ishiwata-2013} proved that there is a threshold blow-up modulus concerning the blow-up in finite time. By employing the modified potential well method, Ikehata, Ishiwata and Suzuki \cite{Ikehata-2010} gave a threshold result for the global existence and blow-up of solutions of \eqref{1.1} in the low initial energy level. Subsequently, Fang and Zhang \cite{Fang-Zhang} extended their result to any initial energy level. We refer to \cite{Marek-2006, Quittner-Souplet-2019} for related studies in the supercritical case.

The potential well method
is commonly employed to study parabolic problems on bounded regions \cite{Li-Liu-20018,Xu20018,Payne-1975}, as the Poincar\'{e} inequality can be ensure the existence of a non-empty potential well.
To the best of our knowledge, in the energy subcritical case $p<\frac{n+2}{n-2}$, there are currently no results on blow-up criteria for equation \eqref{1.1} obtained by the potential well method. 
If one attempts to define the potential well for equation \eqref{1.1} in analogy with the bounded domain case \cite{Payne-1975}, the scaling invariance implies that the depth of the potential well is zero. This is because the scaling transformation can be used to construct a sequence of functions on the Nehari manifold whose energy approaches zero. Consequently, the classical potential well theory fails to provide a criterion for initial data of low energy. It offers no information beyond the well-known blow-up result for negative energy \cite{Levine-1973}, leaving the case of low non-negative energy open.

To overcome this problem, inspired by \cite{Kavian,Ikehata-2010}, we employ the forward transformation to convert \eqref{1.1} into a new parabolic equation. This new equation has a key advantage through the introduction of weighted Sobolev spaces. Specifically, the potential well defined in these spaces has a positive depth. By analyzing the dynamics of the new transformed equation, we establish a new criterion for both global existence and finite-time blow-up of solutions to the original equation \eqref{1.1}. Our main result is presented in Theorem \ref{global} in the next section.

This paper is organized as follows. In Section \ref{preliminaries}, we give the transformation and define some weighted spaces, and then present our main theorem. In Section \ref{HU}, we give the complete proof of our main result.

\section{Preliminaries and main results}\label{preliminaries}
We first introduce the forward similarity transform.
Following Kavian \cite{Kavian}, we define the forward similarity variables
\begin{equation}\label{fs40}
s=\log(1+t),\ \ \ y=\frac{x}{\sqrt{t+1}},
\end{equation}
 and define the rescaled functions
\begin{equation}\label{fs4}
w(y,s)=e^{\frac{s}{p-1} }u(e^{\frac{s}{2}}y,e^s-1).
\end{equation}
Then \eqref{1.1} can be rewritten in the form
\begin{equation}\label{0.1}
\begin{aligned}
\left\{
\begin{array}{ll}
w_s+Lw=|w|^{p-1}w+\frac{1}{p-1}w, &y\in\mathbb{R}^n,\ s>0,\vspace{1ex}\\
w(y,0)=u_0(y),&y\in\mathbb{R}^n,
\end{array}
\right.
\end{aligned}
\end{equation}
where the operator $L$ defined on $H^2(K)$ (see below) by
\begin{equation}\label{operator}
Lf:=-\Delta f-\frac{y\cdot \nabla f}{2}=-K^{-1}\nabla\cdot(K\nabla f),\ K:= K(y)=\exp(|y|^2/4).
\end{equation}
 We now define some weighted spaces as follows. For $1\le p<\infty$, we put
$$
L^p(K):=\bigg\{f\in L^p(\mathbb{R}^n):\int_{\mathbb{R}^n}|f(y)|^pK(y)dy<\infty\bigg\},
$$
$$ ||f||^p_{L^{p}(K)}:=\int_{\mathbb{R}^n}|f(y)|^pK(y)dy,$$
and
\begin{equation}\label{wei}
H^1(K):=\bigg\{f:f\in L^2(K), \nabla f\in L^2(K)\bigg\},
\end{equation}
$$
||u||^2_{H^{1}(K)}:=\int_{\mathbb{R}^n}(|\nabla u(y)|^2+|u(y)|^2)K(y)dy.
$$
Analogously, we define
$H^2(K):=\{f:f\in H^1(K), \nabla f\in H^1(K)\}.
$
We next recall the following result.
\begin{lemma}[\cite{Kavian}, Lemma 2.1]\label{kavian}
 For $n\ge3$, there holds \\
$(\mathrm{i})$ $u\in H^1(K)\Leftrightarrow K^{1/2}u\in H^1(\mathbb{R}^n)$.\\
$(\mathrm{ii})$ $H^1(K)\hookrightarrow L^q(K)$ for any $q\in[2,\frac{2n}{n-2}]$.\\
$(\mathrm{iii})$ For any $u\in H^1(K),$ there holds $\frac{n}{2}\int_{\mathbb{R}^n} | u|^2K(y)dy\le \int_{\mathbb{R}^n} |\nabla u|^2K(y)dy$.
\end{lemma}
We denote by $T_s\in(0,\infty]$ the maximal existence time of the solution to equation \eqref{0.1}. By \eqref{fs40}, we know that 
\begin{equation}\label{Hong}
T_s=\log(1+T).
\end{equation}
For $w\in H^1(K)$, we define the energy functional
\begin{equation}\nonumber
\begin{aligned}
E(w)
&=\frac{1}{2}\int_{\mathbb{R}^{n}} |\nabla w|^{2}K(y)\mathrm{d}y-\frac{1}{2(p-1)}\int_{\mathbb{R}^{n}}|w|^{2}K(y)\mathrm{d}y-\frac{1}{p+1}\int_{\mathbb{R}^{n}}|w|^{p+1}K(y)\mathrm{d}y.
\end{aligned}
\end{equation}
Multiply the differential equations in \eqref{0.1} by $w_s$  (in the sense of $L^2(K)$),  then integrating on $(0,t)$ we find that
the energy functional satisfies the identity
\begin{equation}\label{yb}
E(u_0)-E(w(s))=\int_0^s\int_{\mathbb{R}^{n}}|w_s(y,\tau)|^2K(y)\mathrm{d}y \mathrm{d}\tau.
\end{equation}
Thus $E(w(s))$ is nonincreasing in $s$.
We define on $H^1(K)$ the Nehari functional $I$ by
\begin{equation}\nonumber
\begin{aligned}
I(w)
&=\int_{\mathbb{R}^{n}} |\nabla w|^{2}K(y)\mathrm{d}y-\frac{1}{p-1}\int_{\mathbb{R}^{n}}|w|^{2}K(y)\mathrm{d}y-\int_{\mathbb{R}^{n}}|w|^{p+1}K(y)\mathrm{d}y.
\end{aligned}
\end{equation}
The potential well associated with problem \eqref{0.1} is the set
\begin{equation}\nonumber
W:=\{w\in H^1(K): E(w)< d,\ I(w)>0\}\cup\{0\},
\end{equation}
where $d$ is the depth of the potential well defined as follows:
\begin{equation}\label{depth}
d:=\inf\{E(w):w\in H^1(K) \backslash\{0\},\ I(w)=0\}.
\end{equation}
The exterior of the potential well is the set
\begin{equation}\nonumber
Z:=\{w\in H^1(K): E(w)< d,\ I(u)<0\}.
\end{equation}

Ikehata, Ishiwata and Suzuki \cite{Ikehata-2010} (Proposition 1.3 and 1.4) established the global existence and blow-up result for the critical case. In the below theorem, we extend their 
result to the case $\frac{n+2}{n}<p<\frac{n+2}{n-2}$.
The following is our main result.
\begin{theorem}\label{global}
Let $1+\frac{2}{n}<p<\frac{n+2}{n-2}$ and $u_0\in  H^1(K)$. 
\begin{enumerate}
    \item If $u_0 \in W$, then the solution of \eqref{1.1} exists globally.
    \item If $u_0 \in Z$, then the solution of \eqref{1.1} blows up in finite time.
\end{enumerate}
\end{theorem}

The rest of this paper is giving the proof of the above theorem.

\section{Global existence and finite time blow-up}\label{HU}
Before we present the proof of Theorem \ref{global}, we first
show that the sets $W$ and $Z$ are nonempty. In the rest of this paper, $C$ denotes a generic positive constant, which may change from line to line.
\begin{proposition}
As defined in \eqref{depth}, the depth $d$ of the potential well is positive.
\end{proposition}
\begin{proof}
Take $w\in H^1(K) \backslash\{0\}$ satisfying $I(w)=0$, then
\begin{equation}\label{FS}
\begin{aligned}
\int_{\mathbb{R}^{n}} |\nabla w|^{2}K(y)\mathrm{d}y-\frac{1}{p-1}\int_{\mathbb{R}^{n}}|w|^{2}K(y)\mathrm{d}y
=\int_{\mathbb{R}^{n}}|w|^{p+1}K(y)\mathrm{d}y.
\end{aligned}
\end{equation}
We denote $$ {A}:=\int_{\mathbb{R}^{n}}|w|^{p+1}K(y)\mathrm{d}y.$$
Since $p>1+\frac{2}{n}$, by Lemma \ref{kavian} (iii), we have
\begin{equation}\label{nb}
    {A}=\int_{\mathbb{R}^{n}} |\nabla w|^{2}K(y)\mathrm{d}y-\frac{1}{p-1}\int_{\mathbb{R}^{n}}|w|^{2}K(y)\mathrm{d}y> C\int_{\mathbb{R}^{n}}|w|^{2}K(y)\mathrm{d}y.
\end{equation}
From $1+\frac{2}{n}<p<\frac{n+2}{n-2}$, by \eqref{FS}, \eqref{nb} and Lemma \ref{kavian} (ii), (iii), we have
\begin{equation}\nonumber
\begin{aligned}
{A}&\le C\bigg(\int_{\mathbb{R}^{n}}|\nabla w|^{2}K(y)\mathrm{d}y\bigg)^{\frac{p+1}{2}}= C\bigg({A}+\frac{1}{p-1}\int_{\mathbb{R}^{n}}|w|^{2}K(y)\mathrm{d}y\bigg)^{\frac{p+1}{2}}
\le C{A}^{\frac{p+1}{2}}. \end{aligned}
\end{equation}
Thus
$$
{A}\ge C^{\frac{-2}{p-1}}:=C_0>0.
$$
By \eqref{FS}, we get
$$
E(w)=\frac{(p-1){A}}{2(p+1)}\ge \frac{(p-1)C_0}{2(p+1)}>0.
$$
Then by the definition of $d$ in \eqref{depth}, we know that $d>0$.
\end{proof}

\begin{proposition}
The sets $W$ and $Z$ are nonempty and invariant along the flow of \eqref{1.1}.
\end{proposition}
\begin{proof}
We first prove that $W$ and $Z$ are nonempty. For any $w\in H^1(K)$, we know that $\underset{b\to 0^+}\lim E(bw)=0$ and $\underset{b\to +\infty}\lim E(bw)=-\infty$,
Therefore,  for $0<b\ll1$ or $b\gg1$, we have $E(bw)<d$.
On the other hand, we see that  $I(bw)>0$ for $0<b\ll1$, and $I(bw)<0$ for $b\gg1$. We thus obtain that $bw\in W$ for $0<b\ll1$ and $bw\in {Z}$ for $b\gg1$, i.e., the sets $W$ and $Z$ are nonempty.

We next show that $W$ is an invariant set of problem \eqref{1.1}. If $u_0 \in W$ and $u_0\neq0$, we have $E(u_0)<d$ and $I(u_0)>0$. Then $E(s)$ is nonincreasing in $s$ implies
\begin{equation}\label{kh}
E(w(s))\le E(u_0)<d, \ \ \ \ \ \ s\ge 0.
\end{equation}
We claim that $I(w(s))>0$ for any $s\ge0$. Otherwise, by the continuity of $I(w)$, there exists $\overline{s}>0$ such that $I(w(\overline{s}))=0$, which implies $E(w(\overline{s}))\ge d$. It is a contradiction with \eqref{kh}. The claim follows. Combining \eqref{kh}, we know that $W$ is an  invariant set  of equation \eqref{1.1}.
Similarity, we can show that $Z$ is an  invariant set of \eqref{1.1}.

This completes the proof.
\end{proof}

\begin{lemma}\label{dd}
For any $\varepsilon>0$, there holds $$d\le d_{\varepsilon}+\frac{\varepsilon}{2},$$ where
$d_{\varepsilon}$ is defined by
\begin{equation}\label{dd1}
d_{\varepsilon}:=\inf\{E(w): w\in H^1(K),\ I(w)=-\varepsilon\}.
\end{equation}
\end{lemma}
\begin{proof}
From $I(w)=-\varepsilon$, by direct computation we have $d_{\varepsilon}>-\infty$.
For any fixed $\varepsilon>0$, we choose a sequence $\{w_j\}\in H^1(K)$ such that
$$I(w_j)=-\varepsilon\ \ \text{and}\ \ E(w_j)\to d_{\varepsilon}\ \ \text{as}\ \ j\to\infty.$$
We have $I(bw_j)>0$ for $b>0$ sufficiently small. Since $I(w_j)=-\varepsilon<0$, then by continuity, there exists a sequence $\{a_j\}$ satisfying $0<a_j<1$ such that $I(a_jw_j)=0$. From $I(w_j)=-\varepsilon$, we have
\begin{equation}\label{dmd}
\begin{aligned}
\varepsilon+\int_{\mathbb{R}^{3}} |\nabla w_j|^{2}K(y)\mathrm{d}y=\frac{1}{p-1}\int_{\mathbb{R}^{n}}|w_j|^{2}K(y)\mathrm{d}y+\int_{\mathbb{R}^{n}}|w_j|^{p+1}K(y)\mathrm{d}y.
\end{aligned}
\end{equation}
Since $I(a_ju_j)=0$, combining \eqref{dmd} we get
\begin{equation}\nonumber
\begin{aligned}
&d\le E(a_jw_j)\\
&=\frac{a_j^2}{2}\int_{\mathbb{R}^{n}} |\nabla w_j|^{2}K(y)\mathrm{d}y-\frac{a_j^2}{2(p-1)}\int_{\mathbb{R}^{n}}|w_j|^{2}K(y)\mathrm{d}y 
-\frac{a_j^{p+1}}{p+1}\int_{\mathbb{R}^{n}}|w_j|^{p+1}K(y)\mathrm{d}y\\
&=E(w_j)+\frac{a_j^2-1}{2}\int_{\mathbb{R}^{n}} |\nabla w_j|^{2}K(y)\mathrm{d}y+\frac{1-a_j^2}{2(p-1)}\int_{\mathbb{R}^{n}}|w_j|^{2}K(y)\mathrm{d}y\\
&\ \ \ \ \ +\frac{1-a_j^{p+1}}{p+1}\int_{\mathbb{R}^{n}}|w_j|^{p+1}K(y)\mathrm{d}y\\
&= E(w_j)+\frac{1-a_j^2}{2}\varepsilon+\frac{(p+1)a_j^2-2a_j^{p+1}-p+1}{2(p+1)}\int_{\mathbb{R}^{n}}|w_j|^{p+1}K(y)\mathrm{d}y\\
&\le  E(w_j)+\frac{\varepsilon}{2}.
\end{aligned}
\end{equation}
Here we use the fact that $(p+1)a_j^2-2a_j^{p+1}-p+1\le0$ for $a_j\in(0,1)$.
We conclude the result by letting $j\to\infty$ in the right-hand side of the above inequality.
\end{proof}

We next present the proof of our main result.
\begin{proof}[Proof of Theorem \ref{global}]
$(\mathrm{i})$ Since $u_0\in W$, we obtain $w(s)\in W$ for $s\in[0,T_s)$. By $I(w(s))>0$, we have
\begin{equation}\label{ip}
\begin{aligned}
\int_{\mathbb{R}^{n}} |\nabla w|^{2}K(y)\mathrm{d}y>\frac{1}{p-1}\int_{\mathbb{R}^{n}}|w|^{2}K(y)\mathrm{d}y+\int_{\mathbb{R}^{n}}|w|^{p+1}K(y)\mathrm{d}y.
\end{aligned}
\end{equation}
Then we get
\begin{equation}\label{DC}
\begin{aligned}
d>E(w)&=\frac{1}{2}\int_{\mathbb{R}^{n}} |\nabla w|^{2}K(y)\mathrm{d}y-\frac{1}{2(p-1)}\int_{\mathbb{R}^{n}}|w|^{2}K(y)\mathrm{d}y \\
&\ \ \ \ -\frac{1}{p+1}\int_{\mathbb{R}^{n}}|w|^{p+1}K(y)\mathrm{d}y\\
&> \frac{p-1}{2(p+1)}\int_{\mathbb{R}^{n}} |\nabla w|^{2}K(y)\mathrm{d}y -\frac{1}{2(p+1)}\int_{\mathbb{R}^{n}}|w|^{2}K(y)\mathrm{d}y.
\end{aligned}
\end{equation}
From Lemma \ref{kavian} (iii) and $p>1+\frac{2}{n}$, there exists a sufficiently small constant $c>0$ such that
\begin{equation}\label{DC1}
\frac{p-1}{2(p+1)}\int_{\mathbb{R}^{n}} |\nabla w|^{2}K(y)\mathrm{d}y -\frac{1}{2(p+1)}\int_{\mathbb{R}^{n}}|w|^{2}K(y)\mathrm{d}y\ge c\int_{\mathbb{R}^{n}} |\nabla w|^{2}K(y)\mathrm{d}y.
\end{equation}

Combining \eqref{DC} and \eqref{DC1}, there holds that $||w(s)||_{H^1(K)}$ is uniformly bounded for $s\in [0,T_s)$, then by the well-posedness of \eqref{0.1} in $H^1(K)$ (see \cite{Kavian} or Example 51.24 of \cite{Quittner-Souplet-2019}), we know that the maximal existence time of problem \eqref{0.1} can be extended to $\infty$. 
Then we have $T=\infty$ by \eqref{Hong}.


$(\mathrm{ii})$ Since $u_0\in Z$, we have $E(u_0)<d$ and $I(u_0)<0$. Fix $\varepsilon_1>0$ such that
$$\varepsilon_1<\min\big(-I(u_0),d-E(u_0)\big).$$
By Lemma \ref{dd}, we have
\begin{equation}\label{lwF}
E(w(s))\le E(u_0)<d-\varepsilon_1<d_{\varepsilon_1},\ \ \forall s\in [0, T_s),
\end{equation}
where $d_{\varepsilon_1}$ is defined in \eqref{dd1}. By \eqref{lwF} and $I(u_0)<-\varepsilon_1$, it follows by continuity that, for any $s\in[0,T_s)$, there holds $I(w(s))<-\varepsilon_1$.

 The following process is based on the concave argument by Levine \cite{Levine-1973}.
Let $M(s)=\frac{1}{2}\int_0^s\int_{\mathbb{R}^{n}}  |w(y,\tau)|^2K(y)dyd\tau.$ From Lemma \ref{kavian} $(\mathrm{iii})$,
we have
\begin{equation}\label{a01}
\begin{aligned}
M''(s)&=\int_{\mathbb{R}^{3}} ww_sK(y)dy=-I(w(s))\\
&=-\int_{\mathbb{R}^{n}} |\nabla w|^{2}K(y)\mathrm{d}y+\frac{1}{p-1}\int_{\mathbb{R}^{n}}|w|^{2}K(y)\mathrm{d}y+\int_{\mathbb{R}^{n}}|w|^{p+1}K(y)\mathrm{d}y\\
&=-(p+1) E(w(s))+\frac{p-1}{2}\int_{\mathbb{R}^{n}}|\nabla w|^2K(y)dy-\frac{1}{2}\int_{\mathbb{R}^{n}}|w|^2K(y)dy\\
&\ge -(p+1) E(w(s))+\frac{n(p-1)-2}{4}\int_{\mathbb{R}^{n}}|w|^2K(y)dy.
\end{aligned}
\end{equation}
Assume for contradiction that $T_s=\infty$. By $M''(s)=-I(w(s))>\varepsilon_1,$ we have
$\lim_{s\to\infty}M'(s)=\infty$. Then there exists $s'>0$ such that
\begin{equation}\label{gk}
\frac{n(p-1)-2}{4}\int_{\mathbb{R}^{n}}|w(s)|^2K(y)dy>(p+1)E(u_0),\ \ \forall s\ge s'.
\end{equation}
Then by \eqref{yb}, \eqref{a01} and \eqref{gk},  we get
$$
M''(s)\ge(p+1)\int_0^s\int_{\mathbb{R}^{n}}|w_s(y,\tau)|^2K(y)\mathrm{d}y\mathrm{d}\tau\ \ \forall s\ge s'.
$$
Hence for any $s\ge s'$, by the
Cauchy-Schwarz inequality, there holds
\begin{equation}\label{to}
\begin{aligned}
M(s)M''(s)&\ge \frac{p+1}{2}\left(\int_0^s\int_{\mathbb{R}^{n}} |w(y,\tau)|^2K(y)dyd\tau\right)\left(\int_0^s\int_{\mathbb{R}^{n}}|w_s(y,\tau)|^2K(y)\mathrm{d}y\mathrm{d}\tau\right)\\
&\ge \frac{p+1}{2}\left( \int_0^s\int_{\mathbb{R}^{n}}w_s(y,\tau)w(y,\tau)\mathrm{d}y\mathrm{d}\tau\right)^2\\
&= \frac{p+1}{2}\big(M'(s)-M'(0)\big)^2.
\end{aligned}
\end{equation}
By $\lim_{s\to\infty}M'(s)=\infty$, there exist $\alpha>0$ and $s''\ge s'$ such that
$$ M(s)M''(s)\ge  \frac{p+1}{2}\big(M'(s)-M'(0)\big)^2\ge (1+\alpha)(M'(s))^2,\ \ \forall s\ge s''.
$$

Denote $Y(t)=M^{-{\alpha}}(t)$, the above inequality implies that $Y(t)$ is concave in $[s'',\infty)$. But this contradicts that $Y(t)$ is positive and decreasing in $[s'',\infty)$.  We thus have $T_s<\infty$, and hence $T<\infty$ by \eqref{Hong}.

The proof is completed.
\end{proof}
\section*{Disclosure statement}
No potential conflict of interest was reported by the authors.

\section*{Data Availability Statement}
Data sharing is not applicable to this paper as no datasets were generated or analysed during the current study.

\section*{Acknowledgements}
K. Zhang would like to thank Prof. Charles Collot and Prof. Li Ma for their helpful discussions. This manuscript was
written while K. Zhang was visiting the Laboratoire Analyse, G$\mathrm{\acute{e}}$om$\mathrm{\acute{e}}$trie et Mod$\mathrm{\acute{e}}$lisation
(UMR CNRS 8088) of CY Cergy Paris Universit$\mathrm{\acute{e}}$, which is acknowledged for the hospitality. K. Zhang is supported by the China Scholarship Council [No. 202206460045]. Z. Li is supported by the Doctoral Introduction Fund Project of Shandong Technology and Business University (No. BS202462).

\end{document}